\documentclass{cimart}

\usepackage{xcolor, amstext, amsfonts, amssymb, amsbsy, latexsym}
\usepackage{enumerate}
\usepackage[T1]{fontenc}
\usepackage{xy, hhline}

\newtheorem{lemmaArt}{Lemma}[section]
\newtheorem{theoArt}[lemmaArt]{Theorem}
\newtheorem{propArt}[lemmaArt]{Proposition}

\newtheorem{conjArt}[lemmaArt]{Conjecture}
\theoremstyle{definition}
\newtheorem{definArt}[lemmaArt]{Definition}
\newtheorem{exampleArt}[lemmaArt]{Example}

\newenvironment{eq}{\begin{equation}}{\end{equation}}

\newcommand{\rf}[1]{(\ref{#1})}

\newcommand{\Char}{\mathop{\rm char}}
\newcommand{\id}[1]{{{\rm id}\{{#1}\}}}

\newcommand{\FF}{\mathbb{F}}

\newcommand{\ZZ}{\mathbb{Z}}

\newcommand{\NN}{\mathbb{N}}
\newcommand{\tq}{ \ | \ }
\newcommand{\st}{ \ | \ }
\newcommand{\X}{\langle X\rangle}
\newcommand{\FX}{\FF\langle X\rangle}
\newcommand{\Fxy}{\FF\langle x,y\rangle}

\newcommand{\A}{\mathsf{A}}
\newcommand{\V}{\mathsf{V}}
\newcommand{\I}{\mathcal{I}}
\newcommand{\B}{\mathsf{B}}

\newcommand{\W}{\mathsf{W}}

\newcommand{\algA}{\mathcal{A}}
\newcommand{\algB}{\mathcal{B}}

\newcommand{\algV}{\mathcal{V}}

\newcommand{\LA}{\langle}
\newcommand{\RA}{\rangle}
\newcommand{\eqPI}{\sim_{\rm PI}}

\newcommand{\si}{\sigma}
\newcommand{\al}{\alpha}
\newcommand{\be}{\beta}
\newcommand{\ga}{\gamma}

\newcommand{\de}{\delta}
\newcommand{\De}{\Delta}
\newcommand{\Ga}{\Gamma}
\newcommand{\Id}[1]{{{\rm Id}({#1})}}
\newcommand{\IdF}[1]{{{\rm Id}_{\FF}({#1})}}

\newcommand{\un}[1]{{\underline{#1}} }

\newcommand{\mdeg}{\mathop{\rm mdeg}}

\newcommand{\mw}{\mathop{\rm mw}}      % monomial weight
\newcommand{\bw}{\mathop{\rm bw}}      % bracket weight

\newcommand{\Sym}{\mathcal{S}} % symmetric group

\newcommand{\lin}{\mathop{\rm lin}}

\newcommand{\St}{{\rm St}}

\title{% Please, capitalize only the first word
    Weak polynomial identities of small degree for the Weyl algebra
    }

\author{% Please, use "Firstname Lastname" format, without abreviations
    Artem Lopatin, Carlos Arturo Rodriguez Palma and Liming Tang
    }

\authorinfo[% Format: please, use "F. Name"
    A. Lopatin]{% Format: please, use "Department (only if 
    Universidade Estadual de Campinas (UNICAMP), Campinas, SP, Brazil}{%
    dr.artem.lopatin@gmail.com
    }

\authorinfo[% Format: please, use "S. Name"
    Carlos Arturo Rodriguez Palma]{% Format: please, use "Department (only if 
    Universidad Industrial de Santander, Bucaramanga, Santander, Colombia}{%
    caropal@uis.edu.co
    }

\authorinfo[% Format: please, use "T. Name"
    Liming Tang]{% Format: please, use "Department (only if 
    School of Mathematical Sciences, Harbin Normal University,  Harbin, China}{%
    limingtang@hrbnu.edu.cn
    }

\abstract{%
    In this paper we investigate weak polynomial identities for the Weyl algebra $\mathsf{A}_1$ over an infinite field of arbitrary characteristic. Namely, we describe weak polynomial identities of the minimal degree, which is three, and of degrees 4 and 5. We also describe  weak polynomial identities in two variables.
    }

\keywords{% 2-5 keywords
    Polynomial identities, Weak polynomial identities, Matrix identities, Weyl algebra,  Positive characteristic.
    }

\msc{% Format: 1 code (primary); 0-4 codes (secondary).  See
    16R10.
    }

\VOLUME{33}
\YEAR{2025}
\ISSUE{1}
\NUMBER{1}
\DOI{https://doi.org/10.46298/cm.13107}

\begin{document}

% Here is where the main text should be typed:

% A table of contents will be automatically inserted in your article if it
% has 3 or more sections.  Please, do not try to manually change this
% behaviour.

% Also, please consider the following suggestions while preparing your 
% manuscript (as they will speed up the editorial process):
% * Avoid starting a new sentence with a mathematical formula;
% * Try to separate adjacent formulas with words;
% * Avoid inline formulas longer than half of a line. You can use math 
%   displays (\[...\]) instead;
% * Consider the use the enumerate and itemize environments for lists;
% * Consider the use of \dots, \ldots, \dotsc, \cdot, etc, instead of "..." 
%   or ".";
% * Instead of numbering or citing an article by hand (using parenthesis or 
%   brackets), consider the use of \cite, \ref and \eqref for citations and
%   cross-references;
% * Try to avoid inserting horizontal or vertical spacing, such as \hskip, 
%   \vskip and \bigskip;
% * Try to avoid inserting line or page brakes, such as \\, \newpage and
%   \clearpage.

% Acknowledgments should be added at the end of this section (right before
% the refences section) as a \subsection* (a subsection without a number):
% \subsection*{Acknowledgments} ...

%========================================================
%S1======================================================
\section{Introduction}\label{section_intro}

Assume that $\FF$ is an infinite field of arbitrary characteristic $p=\Char\FF\geq0$. All vector spaces and algebras are over $\FF$  and all algebras are associative, unless stated otherwise. We write $\FF\LA x_1,\ldots,x_n\RA$ for the free unital $\FF$-algebra with free generators  $x_1,\ldots,x_n$. In case the free generators are $x_1,x_2,\dotsc$ the corresponding free algebra is denoted by  $\FF\LA X\RA$.

A polynomial identity for a unital $\FF$-algebra $\algA$ is an element $f(x_1,\ldots,x_m)$ of $\FF\LA X\RA$ such that $f(a_1,\ldots,a_m)=0$ in $\algA$ for all $a_1,\ldots,a_m\in \algA$. The set $\IdF{\algA}=\Id{\algA}$ of all polynomial identities for $\algA$ is a T-ideal, i.e.,  $\Id{\algA}$ is an ideal of $\FX$ such that $\phi(\Id{\algA})\subset \Id{\algA}$ for every endomorphism  $\phi$ of $\FX$. Given an $\FF$-subspace $\algV\subset \algA$,  we write $\IdF{\algV}=\Id{\algV}$ for the ideal of all polynomial identities for $\algV$.  Note that $\Id{\algV}$ is an {\it L-ideal} (or {\it weak T-ideal}), i.e.,  $\phi(\Id{\algV})\subset \Id{\algV}$ for every linear endomorphism  $\phi$ of $\FX$, but $\Id{\algV}$ is not a T-ideal in general. We say that a space $\algV$ generates the algebra $\algA$, if any element of $\algA$ can be written as a non-commutative polynomial without free term in some elements of $\algV$. If a space $\algV$ generates the algebra $\algA$, then the polynomial identities for $\algV$ are called {\it weak} polynomial identities for the pair ($\algA,\algV$) and we denote $\Id{\algV}=\Id{\algA,\algV}$.

Weak polynomial identities were introduced in 1973 by Razmyslov~\cite{Razmyslov_1973_AL,Razmyslov_1973_USSR} (see also book~\cite{Razmyslov_book}), who applied them to study polynomial identities of matrices. 
%In characteristic zero case some results are known about weak polynomial identities for the algebra of $n\times n$ matrices $M_n$. 
Razmyslov~\cite{Razmyslov_1973_AL}, Drensky~\cite{Drensky_1997} and  Koshlukov~\cite{Koshlukov_1997} described weak polynomial identities for the pair $(M_2,sl_2)$ over a field of an arbitrary characteristic, where $sl_2$ is the space of all traceless matrices. Weak polynomial identities of small degrees for the pair $(M_3, sl_3)$ were studies by Drensky, Rashkova~\cite{Drensky_Rashkova_1993} and by  Blachar, Matzri, Rowen, Vishne~\cite{Blachar_Matzri_Rowen_Vishne_2021}. 

For $p=0$ weak polynomial identities for the pair $(M_2,H_2)$ were described by Drensky~\cite{Drensky_1986}, where $H_n$ stands for the space of all symmetric $n\times n$ matrices. Minimal weak polynomial identities for the pair $(M_n,H_n)$ for an arbitrary $n>1$ were described by Ma and Racine~\cite{Ma_Racine_1990} in case the characteristic of $\FF$ satisfies certain restrictions. 

Weak polynomial identities were also considered in~\cite{Drensky_1987,  Isaev_Kislitsin_2013, Isaev_Kislitsin_2017, Kislitsin_2022, Koshlukov_1998}, etc. More details on weak polynomial identities can be found in a recent survey by Drensky~\cite{Drensky_2021}.

The Weyl algebra $\A_1$ is generated by $\V=\FF\text{-}\mathrm{span}\{ x,\; y\}$. In this paper we consider weak polynomial identities for the pair $(\A_1,\V)$. 
In Lemma~\ref{lemma_PIs} we show that the following elements of $\FX$ are weak polynomial identities for $(\A_1,\V)$:
\begin{enumerate}
\item[$\bullet$] $\Ga_m(x_1,\ldots,x_m)=[[x_1,x_2],x_3 \cdots x_m]$ for $m\geq3$,

\item[$\bullet$] $\St_3(x_1,x_2,x_3)=x_1 [x_2, x_3] -  x_2 [x_1, x_3] +  x_3 [x_1, x_2]$,

\item[$\bullet$] $T_4(x_1,\ldots,x_4)=[x_1,x_2] [x_3,x_4] - [x_1,x_3] [x_2,x_4] + [x_2,x_3][x_1,x_4]$,
\end{enumerate}

\noindent{}Denote by $\I$ the ideal of $\FX$ generated by 
$$\Ga_3(x_i, x_j, x_k),\;\; \St_3(x_i, x_j, x_k),\;\; T_4(x_i, x_j, x_k, x_l)$$ 
for all $i,j,k,l>0$. In other words, $\I$ is the L-ideal generated by $\Ga_3$, $\St_3$, and $T_4$. Given $f_1,f_2\in\FX$, we say that $f_1$ and $f_2$ are {\it equivalent} and write $f_1\equiv f_2$ in case $f_1-f_2\in \I$. 

In Theorem~\ref{theo_deg3} we describe weak polynomial identities for $(\A_1,\V)$ of the minimal degree, which is three. In Theorem~\ref{theo_2var} we show that every weak polynomial identity for $(\A_1,\V)$ in two variables lies in $\I$. Moreover, all weak polynomial identities for $(\A_1,\V)$ of degrees 4 and 5 belong to $\I$ by Propositions~\ref{prop_deg4} and~\ref{prop_deg5}. Therefore, we formulate the following conjecture:

%---------------------------------------------------------
\begin{conjArt}\label{conj1}
The ideal of all weak polynomial identities for the pair $(\A_1,\V)$ is equal to $\I$. 
\end{conjArt}

The key definitions are given in Section~\ref{section_def} and some properties are considered in Section~\ref{section_prop}. The proofs are based on the notion of a completely reduced form of elements of $\FX$, which is introduced in Section~\ref{section_red}.

%========================================================
%========================================================	
\section{Definitions and known results}\label{section_def}

%========================================================	
\subsection{\texorpdfstring{Polynomial identities for the Weyl algebra $\A_1$}{Polynomial identities for the Weyl algebra $\A_1$}}

The {\it Weyl algebra} $\A_1$ is the unital associative algebra over $\FF$ generated by letters $x$, $y$ subject to the defining relation $yx=xy+1$ (equivalently, $[y,x]=1$, where $[y,x]=yx-xy$), i.e., $$\A_1=\Fxy/\id{yx-xy-1}.$$

We say that algebras $\algA$, $\algB$ are called PI-equivalent and write $\algA \eqPI \algB$ if $\Id{\algA} =\Id{\algB}$. We say that an L-ideal $I\in\FX$ is generated by $f_1,\ldots,f_k\in\FX$ as an L-ideal, if $I$ is an $\FF$-span of $\{f^{(1)} f_i(g_1,\ldots,g_m)f^{(2)}\}$ for all $f^{(1)},f^{(2)}\in \FX$, all linear combinations $g_1,\ldots,g_m$ of letters $\{x_1,x_2,\ldots\}$, and $1\leq i\leq k$. Obviously, in case $f_i$ is multilinear (see Section~\ref{section_notation} below) we can assume that $g_1,\ldots,g_m$ are letters. 

Assume that $p=0$. It is well-known that the algebra $\A_1$ does not have nontrivial polynomial identities. Nevertheless, some subspaces of $\A_1$ satisfy certain polynomial identities. As an example, Dzhumadil'daev proved that the standard polynomial 
$$\St_N(x_{1},\ldots,x_{N})=\sum_{\sigma\in \Sym_{N}}(-1)^{\sigma}x_{\sigma(1)}\cdots x_{\sigma(N)}$$
is a polynomial identity for $\A_1^{(-,s)} = \FF\text{-}\mathrm{span} \{ a y^s \st a\in \FF[x]\}$
if and only if $N>2s$ (Theorem 1 of~\cite{Askar_2014}). More results on polynomial identities for some subspaces of $n^{\rm th}$ Weyl algebra were obtained in~\cite{Askar_2004, Askar_Yeliussizov_2015}. Considering   $\A_1^{(-,1)}$ with respect to the Lie bracket we obtain a simple Lie algebra $\W_1$, which is called Witt algebra.  The well-known open conjecture claims that all polynomial identities for $\W_1$ follow from the standard Lie identity of degree 5. The $\ZZ$-graded identities for $W_1$ were described by Freitas, Koshlukov and  Krasilnikov~\cite{W1_2015}. Moreover, $\ZZ$-graded identities for the related Lie algebra of the derivations of the algebra of Laurent polynomials were described in~\cite{Fideles_Koshlukov_2023_JA, Fideles_Koshlukov_2023_Camb}. 

The situation is drastically different in case $p>0$. Namely, $\A_1$ is PI-equivalent to the algebra $M_p$ of all $p\times p$ matrices over $\FF$. Moreover,  the Weyl algebra $\A_1$ over an arbitrary associative (but possible non-commutative) $\FF$-algebra $\B$ is PI-equivalent to the algebra $M_p(\B)$ of all $p\times p$ matrices over $\B$ (see Theorem~4.9  of~\cite{Lopatin_Rodriguez_2022} for more general result). Polynomial identities for $\A_1^{(-,s)}$ and other subspaces of $\A_1$ were studied in~\cite{Lopatin_Rodriguez_II, Lopatin_Rodriguez_III}.

%========================================================
\subsection{Notations}\label{section_notation}

An algebra that satisfies a nontrivial polynomial identity is called a PI-algebra. A T-ideal $I$ of $\FF\LA X\RA$ generated by $f_1,\ldots,f_k\in \FF\LA X\RA$ is the minimal T-ideal of $\FF\LA X\RA$ that contains $f_1,\ldots,f_k$. 
%Denote $I=\Id{f_1,\ldots,f_k}$. We say that $f\in \FF\LA X\RA$  is a consequence of $f_1,\ldots,f_k$ if $f\in I$. 
We denote by $\X_m$ and $\X$ the monoids  (with unity) freely generated by the letters $x_1,\ldots, x_m$ and $x_1, x_2, \ldots$, respectively. Given $w\in \X_m$, we write $\deg_{x_i}(w)$ for the number of letters $x_i$ in $w$ and $\mdeg(w)\in\NN_0^m$ for the multidegree $(\deg_{x_1}(w),\ldots,\deg_{x_m}(w))$ of $w$, where $\NN_0=\{0,1,2,\ldots\}$ and $\NN=\{1,2,\ldots\}$. An element $f\in\FX$ is called (multi)homogeneous if it is a linear combination of monomials of the same (multi)degree. Given $f=f(x_1,\ldots,x_m)$ of $\FX$, we write $f=\sum_{\un{\de}\in \NN_0^m} f_{\un{\de}}$ for multihomogeneous components $f_{\un{\de}}$ of $f$ with $\mdeg{f_{\un{\de}}}=\un{\de}$. If $f\in\FX$ is multihomogeneous of multidegree $1^m = (1,\ldots,1)$ ($m$ times), then $f$ is called {\it multilinear}.  For $\un{\de}=(\de_1,\ldots,\de_m)$ we denote $|\un{\de}|=\de_1+\cdots+\de_m$. Given $\un{\de}\in\NN_0^{m}$, we write $\FX_{\un{\de}}$ for all elements of $\FX$ of multidegree $\un{\de}$ and we write $\Id{\algA,\algV}_{\un{\de}}$ for all elements of $\Id{\algA,\algV}$ of multidegree $\un{\de}$.

%========================================================
%S======================================================
\section{Properties}\label{section_prop} 

%========================================================
\subsection{Properties of $\A_1$}\label{section_W1}

Given $a\in \FF[x]$, we write $\partial(a)$ for the usual derivative of a polynomial $a$ with respect to the variable $x$. Using the linearity of derivative and induction on the degree of $a\in\FF[x]$ it is easy to see that
\begin{eq}\label{eq0}
[y,a]=\partial(a) \text{ holds in }\A_1 \text{ for all }a\in \FF[x].
\end{eq}%
The following properties are well-known (for example, see~\cite{Benkart_Lopes_Ondrus_I}):

%---------------------------------------------------------
\begin{propArt}\label{prop_basis1}  
\begin{enumerate}
\item[(a)] $\{x^{i}y^{j}\tq i,j\geq 0\}$ and $\{y^{j}x^{i}\tq i,j\geq 0\}$ are $\FF$-bases for $\A_1$. 

\item[(b)] If $p=0$, then the center ${\rm Z}(\A_1)$ of $\A_1$ is $\FF$; if $p>0$, then ${\rm Z}(\A_1)=\FF[x^{p},y^{p}]$.

\item[(c)] If $p>0$, then $\A_1$ is a free module over ${\rm Z}(\A_1)$ and the set $\{x^{i}y^{j} \tq 0\leq i,j<p\}$ is a basis.

\item[(d)] The algebra $\A_1$ is simple if and only if $p=0$.
\end{enumerate}
\end{propArt}
%\medskip

%========================================================
\subsection{Partial linearizations}\label{section_partl}

Assume $f\in \FX$ is multihomogeneous of multidegree $\un{\de}\in\NN_0^m$. Given $1\leq i\leq m$ and $\un{\ga}\in\NN_0^k$ for some $k>0$ with $|\un{\ga}|=\de_i>0$, the {\it partial linearization} $\lin_{x_i}^{\un{\ga}}(f)$ of $f$ of multidegree $\un{\ga}$ with respect to $x_i$ is the multihomogeneous component of 
$$f(x_1,\ldots,x_{i-1},x_{i}+\cdots+x_{i+k-1},x_{i+k},\ldots,x_{m+k-1})$$ 
of multidegree $(\de_1,\ldots,\de_{i-1},\ga_{1},\ldots,\ga_{k},\de_{i+1},\ldots,\de_{m}$). As an example, $$\rm{lin}_{x_2}^{(2,1)}(x_1^2 x_2^3 x_3^2) = x_1^2 (x_2^2 x_3 + x_2 x_3 x_2 + x_3 x_2^2) x_4^2.$$
The result of subsequent applications of partial linearizations to $f$ is also called a partial linearization of $f$. The {\it complete linearization} $\lin(f)$ of $f$ is the result of subsequent applications of $\lin_{x_1}^{1^{\de_1}},\ldots, \lin_{x_m}^{1^{\de_m}}$ to $f$. 

Since $\FF$ is infinite, it is well-known that the following lemma holds (see also Lemma 2.3 of~\cite{Lopatin_Rodriguez_III}). 

%---------------------------------------------------------
\begin{lemmaArt}\label{lemma_id}  Assume $\algA$  is a unital $\FF$-algebra  and  $\algV\subset \algA$ is an $\FF$-subspace. 
\begin{enumerate}
\item[1.] If $f$ is a polynomial identity for $\algV$, then all partial linearizations of $f$ are also polynomial identities for $\algV$.

\item[2.] Assume that all partial linearizations of a multihomogeneous element  $f$ of  $\FX$ are equal to zero over some basis of $\algV$. Then $f$ is a polynomial identity for $\algV$. 
\end{enumerate}
\end{lemmaArt}

%\medskip
\noindent{}Note that part 1 of Lemma~\ref{lemma_id} does not hold in general for a finite field. As an example, see~\cite{Lopatin_Shestakov_2013} for the case of $f(x_1)=x_1^n$ and 
$$\algA=\algV=\frac{\FX}{\id{g^n\,|\,g\in\FX \text{ without constant term}}}.$$

%========================================================
\section{Identities}\label{section_id}

%---------------------------------------------------------
\begin{lemmaArt}\label{lemma_PIs} The following elements of $\FX$ are weak polynomial identities for $(\A_1,\V)$:
\begin{enumerate}
\item[$\bullet$] $\Ga_m(x_1,\ldots,x_m)=[[x_1,x_2],x_3\cdots x_{m}]$ for all $m\geq3$,

\item[$\bullet$] $\St_3(x_1,x_2,x_3)=x_1 [x_2, x_3] -  x_2 [x_1, x_3] +  x_3 [x_1, x_2]$,

\item[$\bullet$] $T_4(x_1,\ldots,x_4)=[x_1,x_2] [x_3,x_4] - [x_1,x_3] [x_2,x_4] + [x_2,x_3][x_1,x_4]$,
\end{enumerate}
\end{lemmaArt}
\begin{proof}
\noindent{\bf 1.} Since $[x,x]=[y,y]=0$ and $[x,y]=-[y,x]=-1$, we have 
\begin{eq}\label{eq_Z}
[u,v]\in Z(A_1) \text{ for all } u,v\in\V.
\end{eq}%
Thus $\Ga_m\in \Id{\A_1,\V}$.

%\medskip
Since any $g$ from the set $\{\St_3,\; T_4\}$ is multilinear, to show that $g\in\Id{\A_1,\V}$ it is enough to show that $g(u_1,\ldots,u_m)=0$ in $\A_1$ for all $u_1,\ldots,u_m\in\{x,y\}$. Obviously, $g(x,\ldots,x)=g(y,\ldots,y)=0$. 

\medskip
\noindent{\bf 2.} Since $\St_3(x,x,y)=\St_3(x,y,y)=0$, we obtain $\St_3\in\Id{\A_1,\V}$.

\medskip
\noindent{\bf 3.} If  $(u_1,\ldots,u_4)$ is equal to $(x,x,x,y)$ or $(x,y,y,y)$, then $T_4(u_{\si(1)},\ldots,u_{\si(4)})=0$ for all $\si\in S_4$. Similarly to $T_4(x,x,y,y)=0-1+1=0$, we obtain that $T_4(u_{\si(1)},\ldots,u_{\si(4)})=0$ for all $\si\in S_4$ and $(u_1,\ldots,u_4)=(x,x,y,y)$. Thus $T_4\in\Id{\A_1,\V}$.
\end{proof}
%\medskip

%---------------------------------------------------------
\begin{lemmaArt}\label{lemma_deg2}
Any weak identity for the pair $(\A_1,\V)$ of degree $\leq 2$ is zero. 
\end{lemmaArt}
\begin{proof}
By Lemma~\ref{lemma_id} it is enough to show that $f=0$ for every multihomogeneous weak identity $f\in\Id{\A_1,\V}$ of degree $\leq 2$.

If $\mdeg(f)=(\de)$ for $\de\in\{1,2\}$, then $f=\al x_1^{\de}$ for $\al\in\FF$ and equality $f(x)=0$ implies $\al=0$.

Assume that $\mdeg(f)=(1,1)$ and $f=\al x_1 x_2 + \be x_2 x_1$ for   $\al,\be\in\FF$. Then  $f(x,x)=0$ implies that $\al+\be = 0$, i.e.,  $f=\al [x_1,x_2]$. Hence $0=f(x,y)=-\al$ implies $\al=0$.
\end{proof}
%\medskip

%---------------------------------------------------------
\begin{lemmaArt}\label{lemma_Ga}
The L-ideal generated by $\Ga_m$, $\St_3$, $T_4$, where $m\geq3$, coincides with $\I$. 
\end{lemmaArt}
\begin{proof}
For $m>3$ consider 
$$\begin{array}{rcl}
\Ga_m(x_1,\ldots,x_m) & = &[x_1,x_2] x_3 x_4 \cdots x_m -  x_3 x_4 \cdots x_m [x_1,x_2] \\
&\equiv &x_3  [x_1,x_2] x_4 \cdots x_m -  x_3 x_4 \cdots x_m [x_1,x_2] \\
&\vdots & \\\
&\equiv & x_3 x_4 \cdots x_m[x_1,x_2] -  x_3 x_4 \cdots x_m [x_1,x_2] =0. \\
\end{array}
$$
The claim is proven.
\end{proof}

%========================================================
\section{Completely reduced bracket-monomials}\label{section_red}

%---------------------------------------------------------
\begin{definArt}\label{def_bracket_monomial}
A product 
$$x_{t_1}\cdots x_{t_l}[x_{r_1},x_{s_1}] \cdots [x_{r_k},x_{s_k}]$$ 
from $\FX$, where $\un{t}\in \NN^l$, $\un{r},\un{s}\in \NN^k$ for some $l\geq0$, $k>0$ with $r_1<s_1,\ldots, r_k<s_k$, is called a {\it bracket-monomial}.
\end{definArt}
%\smallskip

%---------------------------------------------------------
\begin{lemmaArt}\label{lemma_bracket_monomial}
If two bracket-monomials are equal in $\FX$, then they are the same. In other words, if
$$f=x_{t_1}\cdots x_{t_l}[x_{r_1},x_{s_1}] \cdots [x_{r_k},x_{s_k}] \;\text{ and }\; 
f'= x_{t'_1}\ldots x_{t'_{l'}}[x_{r'_1},x_{s'_1}] \cdots [x_{r'_{k'}},x_{s'_{k'}}]$$
are bracket-monomials and $f=f'$ in $\FX$, then $\un{t}=\un{t}'$, $\un{r}=\un{r}'$, $\un{s}=\un{s}'$. 
\end{lemmaArt}
\begin{proof} Note that $f$ can be written as a linear combinations of $2^k$ pairwise different monomials from $\X$ with coefficients $\pm1$. 

We can assume that $l\geq l'$. Thus, $t_1=t'_1, \ldots, t_{l'}=t'_{l'}$. Therefore, without loss of generality, we may assume that $l'=0$. In case $l>0$ we obtain that $f$ is a linear combination of pairwise different monomials which start with $x_{t_1}$, but $f'$ is a linear combination of pairwise different monomials which start with $x_{r'_1}$ and $x_{s'_1}$, where $r'_1\neq s'_1$; a contradiction. Therefore, $l'=0$.

We have that $f$ is a linear combination of pairwise different monomials which start with $x_{r_1}$ and $x_{s_1}$, but $f'$ is a linear combination of pairwise different monomials which start with $x_{r'_1}$ and $x_{s'_1}$. Hence, $\{r_1,s_1\}=\{r'_1,s'_1\}$ and inequalities $r_1<s_1$,  $r'_1<s'_1$ imply that $r_1=r'_1$ and $s_1=s'_1$.  Therefore, without loss of generality we can assume that $f=[x_{r_2},x_{s_2}] \cdots [x_{r_k},x_{s_k}]$ and  
$f'= [x_{r'_2},x_{s'_2}] \cdots [x_{r'_{k'}},x_{s'_{k'}}]$. Repeating the above reasoning several times we conclude the proof.
\end{proof}

%\medskip
%---------------------------------------------------------
\begin{definArt}\label{def_semired}
\noindent{\bf (a)} A bracket-monomial
\begin{eq}\label{eq_semired}
f= x_{t_1}\cdots x_{t_l} [x_{r_1},x_{s_1}] \cdots [x_{r_k},x_{s_k}]\in\FX,
\end{eq}%
where $\un{t}\in \NN^l$, $\un{r},\un{s}\in \NN^k$ for some $l\geq0$, $k>0$ with $r_1<s_1,\ldots, r_k<s_k$ is {\it semi-reduced} if 
\begin{enumerate}
\item[$\bullet$] $t_1 \leq \cdots \leq t_l$;

\item[$\bullet$] $s_1\leq \cdots \leq s_k$.
 \end{enumerate}

\noindent{\bf (b)} A semi-reduced bracket-monomial $f\in\FX$ defined by~(\ref{eq_semired}) is {\it reduced} if 
\begin{enumerate}
\item[$\bullet$] either $l=0$ or $l\geq1$, $t_l\leq s_1$.
\end{enumerate} 

\noindent{\bf (c)} A reduced bracket-monomial $f\in\FX$ defined by~(\ref{eq_semired}) is {\it completely reduced} if
\begin{enumerate}
\item[$\bullet$] do not exist $1\leq i\neq j\leq k$ with $r_j < r_i < s_i<s_j$.
\end{enumerate} 
\end{definArt}

%\medskip
%---------------------------------------------------------
\begin{exampleArt}\label{ex_fine_canon}
Consider the list of all completely reduced bracket-monomials of multidegree $1^m$:  
\begin{itemize}
\item $m=2$: $[x_1,x_2]$; 

\item $m=3$: $x_1[x_2,x_3]$, $x_2[x_1,x_3]$; 

\item $m=4$: $x_1x_2[x_3,x_4]$, $x_1x_3[x_2,x_4]$,  $x_2x_3[x_1,x_4]$, $[x_1,x_2][x_3,x_4]$,  $[x_1,x_3][x_2,x_4]$;  

\item $m=5$: $x_1 x_2 x_3 [x_4,x_5]$, $x_1 x_2 x_4 [x_3,x_5]$, $x_1 x_3 x_4 [x_2,x_5]$, $x_2 x_3 x_4 [x_1,x_5]$,$x_1 [x_2, x_3] [x_4,x_5]$, 

\hspace{1.29cm} $x_1 [x_2, x_4] [x_3, x_5]$, $x_2 [x_1, x_3] [x_4, x_5]$, $x_2 [x_1, x_4] [x_3, x_5]$, $x_3 [x_1, x_4] [x_2, x_5]$ 
 
\end{itemize}
\end{exampleArt}
\medskip

For $1\leq i<j$ we consider $\NN_0^i$ as a subset of $\NN_0^j$ by 
$$(r_1,\ldots,r_i)\to (r_1,\ldots,r_i,\underbrace{0,\ldots,0}_{j-i}).$$% 
Assume $\un{r}\in\NN_0^i$ and $\un{s}\in\NN_0^j$ for some $i,j\geq1$. Then we write $\un{r}<\un{s}$ for the lexicographical order on $\NN_0^k$, where $k=\max\{i,j\}$ and we consider $\un{r}$, $\un{s}$ as elements of $\NN_0^k$.

%---------------------------------------------------------
%\begin{definArt}\label{defin_order}  Consider the following lexicographical order $\prec$ on the set of all sequences $(\un{r}^{(1)},\ldots,\un{r}^{(l)})$, where  $\un{r}^{(1)}\in \NN_0^{k_1}, \ldots, \un{r}^{(l)}\in \NN_0^{k_l}$ and $l,k_1,\ldots,k_l\geq1$ are arbitrary.  For $R=(\un{r}^{(1)},\ldots,\un{r}^{(l)})$ and $S=(\un{s}^{(1)},\ldots,\un{s}^{(t)})$ we have $R\prec S$  if and only if  one of the following conditions holds:
%\begin{enumerate}
%\item[$\bullet$] $l<t$ and $\un{r}^{(i)}=\un{s}^{(i)}$ for all $1\leq i\leq l$; 
%
%\item[$\bullet$] there exists $1\leq k \leq\min\{l,t\}$ such that $\un{r}^{(i)}=\un{s}^{(i)}$ for all $1\leq i<k$ and $\un{r}^{(k)}< \un{s}^{(k)}$ with respect to the above considered lexicographical order $<$.
%\end{enumerate}
%In case either $R\prec S$ or $R=S$, we write $R\preceq S$.
%\end{definArt}

%---------------------------------------------------------
\begin{definArt} Consider a bracket-monomial 
\begin{eq}\label{eq_defin3}
f=x_{t_1}\cdots x_{t_l}[x_{r_1},x_{s_1}] \cdots [x_{r_k},x_{s_k}]\in\FX,
\end{eq}%
where $\un{t}\in \NN^l$, $\un{r},\un{s}\in \NN^k$ for some $l\geq0$, $k>0$ and $r_1<s_1,\ldots, r_k<s_k$. Then 
\begin{enumerate}
\item[(a)] the {\it monomial weight} of $f$ in case $l>0$ is $\mw(f)=(t_{\si(1)},\ldots,t_{\si(l)})$ for some permutation $\si\in\Sym_l$ such that
$$t_{\si(1)}\geq \cdots \geq t_{\si(l)}$$
\noindent{}and $\mw(f)=(0)$ in case $l=0$.

\item[(b)] the {\it bracket weight} of $f$ is $\bw(f)=(s_{\si(1)}-r_{\si(1)}, \ldots,s_{\si(k)}-r_{\si(k)})$ for some permutation $\si\in\Sym_k$ such that 
$$s_{\si(1)}-r_{\si(1)} \geq \cdots \geq s_{\si(k)}-r_{\si(k)}.$$
%For $f=1$ from $\FX$ we set $\bw(f)=(0)$.
\end{enumerate}
\end{definArt}
%\medskip

%---------------------------------------------------------
\begin{exampleArt} \noindent{\bf (a)} We have $\St_3(x_1,x_2,x_3)= f_1 - f_2 + f_3$ for the semi-reduced bracket-monomials  $$f_1=x_1 [x_2, x_3],\;\; f_2=x_2 [x_1, x_3],\;\; f_3=x_3 [x_1, x_2].$$% 
Then $\mw(f_1)=(1)$, $\mw(f_2)=(2)$, $\mw(f_3)=(3)$ and $\bw(f_1)=(1)$, $\bw(f_2)=(2)$, $\bw(f_3)=(1)$. Note that $f_1$, $f_2$ are reduced, but $f_3$ is not reduced.

%\medskip
\noindent{\bf (b)} We have $T_4(x_1,\ldots,x_4)=h_1 - h_2 +h_3$ for the reduced bracket-monomials $$h_1=[x_1,x_2] [x_3,x_4],\;\;  h_2=[x_1,x_3] [x_2,x_4],\;\; h_3=[x_2,x_3][x_1,x_4].$$% 
Then $\mw(h_i)=(0)$ for $i=1,2,3$ and $\bw(h_1)=(1,1)$, $\bw(h_2)=(2,2)$, $\bw(h_3)=(3,1)$. Note that $h_1$, $h_2$ are completely reduced, but $h_3$ is not completely reduced.
\end{exampleArt}

%\medskip
%---------------------------------------------------------
\begin{lemmaArt}\label{lemma_semi_reduced}
Assume that $f\in\FX$ is multihomogeneous of multidegree $\un{\de}\in\NN_0^m$. Then 
there are semi-reduced bracket-monomials $f_i\in\FX$ and $\al_i,\be\in\FF$ such that
$$f\equiv \be x_1^{\de_1}\cdots x_m^{\de_m} + \sum_i\al_i f_i,$$
where $\mdeg(f_i)=\un{\de}$ for all $i$. Moreover, 
\begin{enumerate}
\item[(a)] if $f\in\Id{\A_1,\V}$, then $\be=0$;

\item[(b)] if $f$ is a bracket-monomial, then $\be=0$ and $\mw(f_i)\leq \mw(f)$ for all $i$.
\end{enumerate}

\end{lemmaArt}
\begin{proof} 
Assume that $f_1,f_2\in\FF\LA x_1,\ldots,x_m\RA$ are multihomogeneous and $1\leq i<j\leq m$. Since
$$f_1 x_j x_i f_2 = f_1 x_i x_j f_2 - f_1 [x_i,x_j] f_2 = 
 f_1 x_i x_j f_2 - f_1 f_2 [x_i,x_j]  - f_1 [[x_i,x_j],f_2],$$
by Lemma~\ref{lemma_Ga} we obtain that 
\begin{eq}\label{eq_appT3}
f_1 x_j x_i f_2 \equiv f_1 x_i x_j f_2 - f_1 f_2 [x_i,x_j].
\end{eq}%
Lemma~\ref{lemma_Ga} also implies that
\begin{eq}\label{eq_appT3_2}
f_1 [x_i,x_j] f_0 f_2 \equiv f_1 f_0 [x_i,x_j] f_2
\end{eq}%
\noindent{}for every $f_0\in\FX$. Since equivalences~\rf{eq_appT3} and \rf{eq_appT3_2} preserve the multidegree, applying formulas~\rf{eq_appT3} and \rf{eq_appT3_2} to $f$ we obtain that $f\equiv\be x_1^{\de_1}\cdots x_m^{\de_m} + \sum_i\al_i f_i$ for some  semi-reduced bracket-monomials $f_i$ and $\be,\al_i\in\FF$, where $\mdeg(f_i)=\un{\de}$. 

Assume $f\in\Id{\A_1,\V}$. Since $[x,x]=0$, we have $f_i(x,\ldots,x)=0$ in $\A_1$ for all $i$. Therefore, $0=f(x,\ldots,x)= \be\, x^{|\un{\de}|}$. Thus $\be=0$. 

If $f$ is a bracket-monomial, then it is easy to see that  $\be=0$ and $\mw(f_i)\leq \mw(f)$ for all $i$.
\end{proof}

%\medskip
%---------------------------------------------------------
\begin{lemmaArt}\label{lemma_reduced}
Consider a semi-reduced bracket-monomial 
$$f=x_{t_1}\cdots x_{t_l}[x_{r_1},x_{s_1}] \cdots [x_{r_k},x_{s_k}]\in\FX.$$
Then $f\equiv\sum_i \al_i f_i$ for some $\al_i\in\FF$, $f_i\in\FX$ such that $f_i$ is a reduced bracket-monomial, $\mdeg(f_i)=\mdeg(f)$, and $\mw(f_i)\leq \mw(f)$  for all $i$.  
\end{lemmaArt}
\begin{proof} Since $f$ is semi-reduced, we have $r_1<s_1,\ldots, r_k<s_k$,  $t_1 \leq \cdots \leq t_l$ and also $s_1\leq \cdots \leq s_k$, where $l\geq0$, $k>0$.

We prove the lemma by induction on $\mw(f)$. If $\mw(f)=(0)$, then $l=0$ and $f$  is reduced.

Assume that $(0)<\mw(f)$ and for every semi-reduced bracket-monomial $f'\in\FX$ with $\mw(f')<\mw(f)$ the statement of this lemma holds. 

Assume that $f$ is not reduced, i.e., $l\geq1$ and $t_l>s_1$.  Using $\St_3$ from Lemma~\ref{lemma_PIs}, we obtain
\begin{eq}\label{eq_T3appl}
x_{t_l} [x_{r_1}, x_{s_1}] \equiv  x_{s_1} [x_{r_1}, x_{t_l}] - x_{r_1} [x_{s_1}, x_{t_l}].
\end{eq}
Thus $f \equiv   f_1  -  f_2$ for 
$$f_1 =  x_{t_1}\cdots x_{t_{l-1}} x_{s_1} [x_{r_1},x_{t_l}][x_{r_2},x_{s_2}] \cdots [x_{r_k},x_{s_k}],$$
$$f_2 =  x_{t_1}\cdots x_{t_{l-1}} x_{r_1} [x_{s_1},x_{t_l}][x_{r_2},x_{s_2}] \cdots [x_{r_k},x_{s_k}].$$
Note that $\mdeg(f_1) = \mdeg(f_2) = \mdeg(f)$. Using part (b) of Lemma~\ref{lemma_semi_reduced}, we obtain semi-reduced bracket-monomials $g_{1i},g_{2j}\in\FX$ and $\al_{1i},\al_{2j}\in\FF$ such that $f_1 \equiv \sum_i \al_{1i}g_{1i}$ and $f_2 \equiv \sum_j \al_{2j}g_{2j}$, where  $\mdeg(g_{1i})=\mdeg(g_{2j})=\mdeg(f)$, $\mw(g_{1i}) \leq \mw(f_1)< \mw(f)$ and  $\mw(g_{2j}) \leq \mw(f_2)< \mw(f)$. Applying the induction hypothesis to $g_{1i}$, $g_{2j}$ we conclude the proof, since $f\equiv \sum_i \al_{1i}g_{1i} - \sum_i \al_{2i}g_{2i}$.
\end{proof}

%\medskip
%---------------------------------------------------------
%\begin{propArt}\label{prop_fine_reduced}
%Assume that $f\in\Id{\A_1,\V}$ is multihomogeneous.  There exists a multihomogeneous $g\in\Id{\A_1,\V}$ in a completely reduced form such that $f\equiv g$ and $\mdeg(f)=\mdeg(g)$.  
%\end{propArt}

%\medskip
%---------------------------------------------------------
\begin{lemmaArt}\label{lemma_fine_reduced}
Consider a reduced bracket-monomial 
$$f= x_{t_1}\cdots x_{t_l} [x_{r_1},x_{s_1}] \cdots [x_{r_k},x_{s_k}]\in\FX.$$  
Then $f\equiv\sum_i \al_i f_i$ for some $\al_i\in\FF$, $f_i\in\FX$ such that $f_i$ is a completely reduced bracket-monomial, $\mdeg(f_i)=\mdeg(f)$, and $\mw(f_i)\leq \mw(f)$ for all $i$.  
\end{lemmaArt}
\begin{proof} Since $f$ is reduced, we have $r_1<s_1,\ldots, r_k<s_k$,  $t_1 \leq \cdots \leq t_l\leq s_1\leq \cdots \leq s_k$, where $l\geq0$, $k>0$.

We prove the lemma by induction on $\mw(f)$. 

\medskip
\noindent{\bf 1.} Assume $\mw(f)=(0)$, i.e.,  $f=  [x_{r_1},x_{s_1}] \cdots [x_{r_k},x_{s_k}]$. To show that the statement of this lemma holds for $f$ we use induction on $\bw(f)$.

Obviously, if $\bw(f)=(1)$, i.e., $f=[x_{r_1},x_{s_1}]$  with $s_1-r_1=1$, then $f$ is completely reduced.

Assume $(1)<\bw(f)$ and for every reduced bracket-monomial $f'\in\FX$ such that $\mw(f')=(0)$ and $\bw(f')<\bw(f)$ the statement of the lemma holds.

Assume that $f$ is not completely reduced, i.e., there are $1\leq i\neq j\leq k$ such that $r_j < r_i < s_i<s_j$.  For short, denote $a_1=r_j$, $a_2=r_i$, $a_3=s_i$, $a_4=s_j$. Note that $a_1<a_2<a_3< a_4$. Using equivalence~(\ref{eq_appT3_2}) we obtain that 
$$ [x_{r_1},x_{s_1}] \cdots [x_{r_k},x_{s_k}] \equiv [x_{a_2},x_{a_3}] [x_{a_1},x_{a_4}]\,  b $$
for some product $b=[x_{r'_1},x_{s'_1}] \cdots [x_{r'_{k-2}},x_{s'_{k-2}}]$ of brackets. Applying the equivalence $T_4(x_{a_1}, x_{a_2}, x_{a_3}, x_{a_4})\equiv0$, we obtain
$$   [x_{a_2},x_{a_3}] [x_{a_1},x_{a_4}] \equiv  
- [x_{a_1},x_{a_2}] [x_{a_3},x_{a_4}] +   [x_{a_1},x_{a_3}] [x_{a_2},x_{a_4}].$$

\noindent{}Thus $f \equiv  - f_1  +  f_2$ for 
$$f_1 =  [x_{a_1},x_{a_2}] [x_{a_3},x_{a_4}] b,$$
$$f_2 =  [x_{a_1},x_{a_3}] [x_{a_2},x_{a_4}] b.$$
Note that $\mdeg(f_1) =\mdeg(f_2) = \mdeg(f)$. Since
$$\bw([x_{a_2},x_{a_3}] [x_{a_1},x_{a_4}])=(a_4-a_1,a_3-a_2)$$
is greater than both $\bw( [x_{a_1},x_{a_2}] [x_{a_3},x_{a_4}] )$ and $\bw( [x_{a_1},x_{a_3}] [x_{a_2},x_{a_4}] )$, 
we can see that $\bw(f_1) < \bw(f)$ and $\bw(f_2) < \bw(f)$. 
  
We use equivalence~(\ref{eq_appT3_2}) to obtain reduced bracket-monomials $g_1$ and $g_2$ such that $g_1 \equiv f_1$ and $g_2 \equiv f_2$, where $\mdeg(g_1)=\mdeg(g_2)=\mdeg(f)$, $\mw(g_1)=\mw(g_2)=(0)$, 
$$\bw(g_1)= \bw(f_1)<\bw(f) \;\;\text{ and }\;\; \bw(g_2)= \bw(f_2)<\bw(f).$$% 
Applying the induction hypothesis to $g_1$ and $g_2$, we can see that the statement of this lemma holds for $f$.

\medskip
\noindent{\bf 2.} Assume that $(0)<\mw(f)$, that is, $l\geq1$, and for every reduced bracket-monomial $f'\in\FX$ with $\mw(f')<\mw(f)$ the claim of this lemma holds. To show that the statement of this lemma holds for $f$ we use induction on $\bw(f)$.

Obviously, if $\bw(f)=(1)$, i.e., $f=x_{t_1}\cdots x_{t_l}[x_{r_1},x_{s_1}]$ with $s_1-r_1=1$, then $f$ is completely reduced.

Assume $(1)<\bw(f)$ and for every reduced bracket-monomial $f'\in\FX$ such that $\mw(f')=\mw(f)$ and $\bw(f')<\bw(f)$ the statement of the lemma holds.

Assume that $f$ is not completely reduced, i.e., there are $1\leq i\neq j\leq k$ such that $r_j < r_i < s_i<s_j$.  For short, denote $a_1=r_j$, $a_2=r_i$, $a_3=s_i$, $a_4=s_j$. Note that $a_1<a_2<a_3< a_4$ and $t_l\leq a_3$. Using equivalence~(\ref{eq_appT3_2}) we obtain that 
$$ [x_{r_1},x_{s_1}] \cdots [x_{r_k},x_{s_k}] \equiv [x_{a_2},x_{a_3}] [x_{a_1},x_{a_4}]\,  b $$
for some product $b=[x_{r'_1},x_{s'_1}] \cdots [x_{r'_{k-2}},x_{s'_{k-2}}]$ of brackets, where $t_l \leq s'_1\leq\cdots\leq s'_{k-2}$. Since $T_4(x_{a_1}, x_{a_2}, x_{a_3}, x_{a_4})\equiv0$, we obtain
$$   [x_{a_2},x_{a_3}] [x_{a_1},x_{a_4}] \equiv  
- [x_{a_1},x_{a_2}] [x_{a_3},x_{a_4}] +   [x_{a_1},x_{a_3}] [x_{a_2},x_{a_4}].$$

\noindent{}Thus $f \equiv  - f_1  +  f_2$ for 
$$f_1 = x_{t_1}\cdots x_{t_l}  [x_{a_1},x_{a_2}] [x_{a_3},x_{a_4}] b,$$
$$f_2 = x_{t_1}\cdots x_{t_l}  [x_{a_1},x_{a_3}] [x_{a_2},x_{a_4}] b.$$
Note that $\mdeg(f_1) =\mdeg(f_2) = \mdeg(f)$. Since
$$\bw([x_{a_2},x_{a_3}] [x_{a_1},x_{a_4}])=(a_4-a_1,a_3-a_2)$$
is greater than $\bw( [x_{a_1},x_{a_2}] [x_{a_3},x_{a_4}] )$ and $\bw( [x_{a_1},x_{a_3}] [x_{a_2},x_{a_4}] )$, 
we can obtain the inequalities $\bw(f_1) < \bw(f)$ and $\bw(f_2) < \bw(f)$. 

\medskip
\noindent{\bf 2.1.} Assume that $l_t\leq a_2$. Then $l_t\leq a_2 < a_3 < a_4$ and $t_l \leq s'_1\leq\cdots\leq s'_{k-2}$. We use equivalence~(\ref{eq_appT3_2}) to obtain reduced bracket-monomials $g_1$ and $g_2$ such that $g_1 \equiv f_1$ and $g_2 \equiv f_2$, where $\mdeg(g_1)=\mdeg(g_2)=\mdeg(f)$, $\mw(g_1)=\mw(g_2)=\mw(f)$, 
$$\bw(g_1)= \bw(f_1) < \bw(f) \;\;\text{ and }\;\; \bw(g_2)= \bw(f_2) < \bw(f).$$  
Applying induction on bracket weight to $g_1$ and $g_2$, we obtain that the statement of the lemma holds for $f$. 

\medskip
\noindent{\bf 2.2.} Assume $a_2<t_l$. Using equivalence~(\ref{eq_T3appl}), we obtain
$f_1 \equiv h_1 - h_2$ for 
$$h_1 = x_{t_1}\cdots x_{t_{l-1}}  x_{a_2}  [x_{a_1},x_{t_l}] [x_{a_3},x_{a_4}] b,$$   
$$h_2 = x_{t_1}\cdots x_{t_{l-1}} x_{a_1}  [x_{a_2},x_{t_l}] [x_{a_3},x_{a_4}] b.$$%
Since $t_l>a_1,a_2$, we have $\mw(h_1) < \mw(f)$ and  $\mw(h_2) < \mw(f)$. Using part (b) of Lemma~\ref{lemma_semi_reduced} and Lemma~\ref{lemma_reduced}, we obtain reduced bracket-monomials $g_{1i'},g_{2j'}\in\FX$ and scalars $\al_{1i'},\al_{2j'}\in\FF$ such that $h_1 \equiv \sum_{i'}\al_{1i'}g_{1i'}$ and $h_2 \equiv \sum_{j'} \al_{2j'}g_{2j'}$, and where  $\mdeg(g_{1i'})=\mdeg(g_{2j'})=\mdeg(f)$, 
$$\mw(g_{1i'}) \leq \mw(h_1)< \mw(f) \;\; \text{ and } \;\; \mw(g_{2j'}) \leq \mw(h_2)< \mw(f).$$%
We apply induction on monomial weight to $g_{1i'}$ and $g_{2j'}$ to show that the statement of the lemma holds for $f_1$.  We establish that the statement of the lemma holds for $f_2$ by repeating the proof from part 2.1. Therefore, the statement of the lemma holds for $f$. 
\end{proof}
%\medskip

%---------------------------------------------------------
\begin{theoArt}\label{theo_canon}
Assume that $f\in \FX$ is multihomogeneous of multidegree $\un{\de}\in\NN_0^m$.  Then 
there are completely reduced bracket-monomials $f_i\in\FX$ and $\al_i,\be\in\FF$ such that
$$f\equiv \be x_1^{\de_1}\cdots x_m^{\de_m} + \sum_i\al_i f_i,$$
where $\mdeg(f_i)=\un{\de}$ for all $i$. Moreover, 
\begin{enumerate}
\item[(a)] if $f\in\Id{\A_1,\V}$, then $\be=0$;

\item[(b)] if $f$ is a bracket-monomial, then $\be=0$ and $\mw(f_i)\leq \mw(f)$ for all $i$.
\end{enumerate}

\end{theoArt}
\begin{proof} Consequently applying Lemmas~\ref{lemma_semi_reduced}, \ref{lemma_reduced}, \ref{lemma_fine_reduced} we obtain the required.
\end{proof}

%========================================================
\section{Minimal weak polynomial identities}\label{section_min}

%---------------------------------------------------------
\begin{theoArt}\label{theo_2var}
Every weak polynomial identity for the pair $(\A_1,\V)$ in two variables lies in the L-ideal $\I$ generated by $\St_3$, $\Ga_3$, $T_4$.
\end{theoArt}
\begin{proof}
Assume that $f\in \FF\LA x_1,x_2\RA$ is a weak polynomial identity in two variables for the pair $(\A_1,\V)$. By Lemma~\ref{lemma_id}, we can assume that $f$ is multihomogeneous of multidegree $(r,s)$ for some $r,s\geq0$. Then Theorem~\ref{theo_canon} implies that $f$ is equivalent to a linear combination of completely reduced bracket-monomials of multidegree $(r,s)$. 

Assume that $r\geq s$. If $s=0$, then $f\equiv0$ by the definition of completely reduced bracket-monomials. Assume $s>0$. Then
$$f(x_1,x_2) \equiv\sum_{i=1}^s \al_i x_1^{r-i} x_2^{s-i} [x_1,x_2]^i$$
for some $\al_i\in\FF$. Since $0=f(x,y)=\sum_{i=1}^s (-1)^i \al_i x^{r-i} y^{s-i}$ in $\A_1$, we obtain  by part (a) of Proposition~\ref{prop_basis1} that $\al_1=\cdots=\al_s=0$, i.e., $f\equiv0$.

The case of $r<s$ can be considered similarly. The proof is completed.
\end{proof}

%---------------------------------------------------------
\begin{lemmaArt}\label{lemma_deg3}
Every weak polynomial identity for the pair $(\A_1,\V)$ of degree $3$ lies in the L-ideal generated by $\Ga_3$ and $\St_3$. 
\end{lemmaArt}
\begin{proof} Assume that $f\in \FX$ is a weak polynomial identity of degree 3 for the pair $(\A_1,\V)$. By Lemma~\ref{lemma_id}, we can assume that $f$ is multihomogeneous of multidegree $\De$ with $|\De|=3$.  By Theorem~\ref{theo_canon},  $f$ is equivalent to a linear combination of completely reduced bracket-monomials of multidegree $\De$. 

Assume $\De=(1,1,1)$. Then
$$f(x_1,x_2,x_3) \equiv \al x_1 [x_2,x_3] + \be x_ 2[x_1,x_3],$$
where $\al,\be\in\FF$. Since $0=f(x,y,y) = - \be y$ and $0=f(x,y,x)= \al x$ in $\A_1$, we obtain $\al=\be=0$. The definition of the ideal $\I$ implies the required.

If $\De=(2,1)$ or $\De=(3)$, then Theorem~\ref{theo_2var} concludes the proof.
\end{proof}

%\medskip
%---------------------------------------------------------
\begin{theoArt}\label{theo_deg3} 
\noindent{\bf 1.} The minimal degree of a non-trivial weak polynomial identity for the pair  $(\A_1,\V)$ is three.

\smallskip
\noindent{\bf 2.} The vector space $\Id{\A_1,\V}_{\De}$ for $|\De|=3$ has the following basis:
\begin{enumerate}
\item[$\bullet$]  $\Ga_3(x_1,x_2,x_3)$, $\Ga_3(x_1,x_3,x_2)$,  and $\St_3(x_1,x_2,x_3)$, in case $\De=1^3$; 

\item[$\bullet$]  $\Ga_3(x_1,x_2,x_1)$, in case $\De=(2,1)$,

\item[$\bullet$]  $\emptyset$, in case $\De=(3)$.
\end{enumerate}
\end{theoArt}
\begin{proof} \noindent{\bf 1.} It follows from Lemmas~\ref{lemma_PIs} and~\ref{lemma_deg2}.

\smallskip
\noindent{\bf 2.} Lemma~\ref{lemma_deg3} implies that $\Ga_3(x_1,x_2,x_1)\neq0$ is a basis for $\Id{\A_1,\V}_{(2,1)}$ and $\emptyset$ is a basis for $\Id{\A_1,\V}_{(2,1)}=\{0\}$.

Since $\Ga_3(x_1,x_2,x_3)$ and $\St_3(x_1,x_2,x_3)$ are multilinear, Lemma~\ref{lemma_deg3} implies that every element $f\in \Id{\A_1,\V}_{1^3}$ lies in the $\FF$-span of 
$$\Ga_3(x_{\si(1)},x_{\si(2)},x_{\si(3)}) \;\text{ and }\; \St_3(x_{\si(1)},x_{\si(2)},x_{\si(3)})$$

\noindent{}for all $\si\in S_3$. Since $\Ga_3(x_2,x_3,x_1) = - \Ga_3(x_1,x_2,x_3) + \Ga_3(x_1,x_3,x_2)$, we obtain that $f$ lies in the $\FF$-span of $\Ga_3(x_1,x_2,x_3)$, $\Ga_3(x_1,x_3,x_2)$, and $\St_3(x_1,x_2,x_3)$. The linear independence follows from straightforward calculations.
\end{proof}

%========================================================
\section{Weak polynomial identities of degrees 4 and 5}\label{section_deg56}

%---------------------------------------------------------
\begin{propArt}\label{prop_deg4}
Any weak polynomial identity for the pair $(\A_1,\V)$ of degree $4$ lies in the L-ideal $\I$ generated by $\Ga_3$, $\St_3$, and $T_4$. 
\end{propArt}
\begin{proof}  Assume that $f\in \FX$ is a weak polynomial identity of degree $4$ for the pair $(\A_1,\V)$. By Lemma~\ref{lemma_id}, we can assume that $f$ is multihomogeneous of multidegree $\De$  with $|\De|=4$. By Theorem~\ref{theo_canon},  $f$ is equivalent to a linear combination of completely reduced bracket-monomials of multidegree $\De$. 

Assume $\De=1^4$. Using Example~\ref{ex_fine_canon} we can see that
\begin{align*}
f(x_1,\ldots,x_4) & \equiv  \al_1\, x_1 x_2 [x_3,x_4]+\al_2\, x_1 x_3 [x_2,x_4]  +  \al_3\, x_2 x_3 [x_1, x_4] \\
&+\be_1 [x_1,x_2][x_3,x_4]+\be_2 [x_1,x_3][x_2,x_4],  
\end{align*}
where $\al_i,\be_j\in\FF$. Since we have $0=f(x,x,y,x)=\al_1 x^2$, $0=f(x,y,x,x) = -\al_2 x^2$, and  $0=f(y,x,x,x)= -\al_3 x^2$, we thus obtain $\al_1=\al_2=\al_3=0$. Then equalities $0=f(x,y,x,y)=\be_1$ and $0=f(x,x,y,y)=\be_2$ imply that $f=0$.

Assume $\De=(2,1,1)$. Then
$$
f(x_1,x_2,x_3) \equiv \al_1\, x_1^2 [x_2,x_3]  +  \al_2\, x_1 x_2 [x_1,x_3]  +  \al_3\, [x_1,x_2][x_1,x_3],
$$
where $\al_1,\al_2,\al_3\in\FF$.  We have $\al_1=0$, since $0=f(x,y,x)=\al_1 x^2$. Thus, the equality $0=f(x,x,y)=-\al_2 x^2$ implies $\al_2=0$. Finally, since $0=f(y,x,x)=\al_3$, we obtain that $f=0$.

If $\De$ belongs to the list $\{(3,1),\; (2,2),\; (4)\}$, then Theorem~\ref{theo_2var} concludes the proof.
\end{proof}
%\medskip

%---------------------------------------------------------
\begin{propArt}\label{prop_deg5}
Any weak polynomial identity for the pair $(\A_1,\V)$ of degree $5$ lies in the L-ideal $\I$ generated by $\Ga_3$, $\St_3$, and $T_4$. 
\end{propArt}
\begin{proof}
Assume that $f\in \FX$ is a weak polynomial identity of degree $5$ for the pair $(\A_1,\V)$. By Lemma~\ref{lemma_id}, we can assume that $f$ is multihomogeneous of multidegree $\De$  with $|\De|=5$.  By Theorem~\ref{theo_canon},  $f$ is equivalent to a linear combination of completely reduced bracket-monomials of multidegree $\De$.

Assume $\De=1^5$. Using Example~\ref{ex_fine_canon} we can see that $f(x_1,\ldots,x_5)$ is equivalent to 
\begin{align*}
\al_1\, x_1 x_2 x_3 [x_4,x_5] + \al_2\, x_1 x_2 x_4 [x_3,x_5] + \al_3\, x_1 x_3 x_4 [x_2,x_5] + \al_4\, x_2 x_3 x_4 [x_1,x_5]\\
+ \be_1\, x_1 [x_2, x_3] [x_4,x_5] + \be_2\, x_1 [x_2, x_4] [x_3, x_5] + \be_3\, x_2 [x_1, x_3] [x_4, x_5] \\
+ \,\be_4\, x_2 [x_1, x_4] [x_3, x_5]+\be_5\, x_3 [x_1, x_4] [x_2, x_5]
\end{align*}
\noindent{}for some $\al_i,\be_j\in\FF$. Considering 
\[
f(x,x,x,y,x)=f(x,x,y,x,x)=f(x,y,x,x,x)=f(y,x,x,x,x)=0,
\]
we obtain $\al_1=\al_2=\al_3=\al_4=0$. Equalities 
\[
f(y,y,x,x,x)=f(y,x,y,x,x)=f(x,y,y,x,x)=f(y,x,x,y,x)=0
\]
imply that $\be_5=\be_4=\be_2=\be_3=0$. Finally, $0=f(x,y,x,y,x)=\be_1 x$ implies $\be_1=0$, i.e., $f=0$.

Assume $\De=(3,1,1)$. Then 
$$f(x_1,x_2,x_3)\equiv \al_1 x_1^3 [x_2,x_3] + \al_2 x_1^2x_2 [x_1,x_3] + \al_3 x_1[x_1,x_2][x_1,x_3]$$
for some $\al_i\in\FF$. We have $\al_1=0$, since $0=f(x,y,x)=\al_1 x^3$. Thus, the equality $0=f(x,x,y)=-\al_2x^3$ implies $\al_2=0$. Finally, since $0=f(y,x,x)=\al_3 y$, we obtain that $f=0$.

Assume $\De=(2,2,1)$. Then $f(x_1,x_2,x_3)$ is equivalent to
$$ \al_1 x_1^2 x_2 [x_2,x_3] + \al_2 x_1x_2^2 [x_1,x_3] + \al_3 x_1[x_1,x_2][x_2,x_3] + \al_4 x_2 [x_1,x_2][x_1,x_3]$$
for some $\al_i\in\FF$. We have $\al_1=\al_3=0$, since $0=f(x,y,x)= \al_1 x^2 y - \al_3 x$. Thus, the equality $0=f(x,x,y)=-\al_2x^3$ implies $\al_2=0$. Finally, since $0=f(y,x,x)=\al_4x$, we obtain that $f=0$.

Assume $\De=(2,1,1,1)$. Then $f(x_1,x_2,x_3,x_4)$ is equivalent to
\begin{align*}
\al_1 x_1^2 x_2[x_3,x_4] + \al_2 x_1^2 x_3 [x_2,x_4] + \al_3 x_1 x_2 x_3 [x_1,x_4]\\
+\,\be_1 x_1 [x_1,x_2] [x_3,x_4] + \be_2 x_1[x_1,x_3][x_2,x_4] +  \be_3 x_2[x_1,x_3][x_1,x_4]
\end{align*}
for some $\al_i,\be_i\in\FF$. Since $0=f(x,x,y,x)=\al_1 x^3$ and $0=f(x,y,x,x)=\al_2x^3$, we have $\al_1=\al_2=0$,. Thus, the equality $0=f(x,x,x,y)= - \al_3x^3$ implies $\al_3=0$. Considering $0=f(y,x,x,x)=\be_3 x$, we obtain $\be_3=0$. Finally, equalities $0=f(y,y,x,x)=\be_2 y$ and $0=f(y,x,y,x)=\be_1 y$ imply that $\be_1=\be_2=0$, i.e., $f=0$. 

If $\De$ belongs to the list $\{(4,1),\; (3,2),\; (5)\}$, then Theorem~\ref{theo_2var} concludes the proof.
\end{proof}
%\medskip

\section*{Acknowledgements}  The first author was supported by FAPESP 2018/23690-6.

%%% REFERENCES %%%
%\bibliography{cimart}

\begin{thebibliography}{99}

\bibitem{Benkart_Lopes_Ondrus_I} G. Benkart, S. Lopes, and M. Ondrus. A parametric family of subalgebras of the Weyl algebra I. Structure and automorphisms. {\it Trans. AMS}, 367(3):1993-2021, 2015.

\bibitem{Blachar_Matzri_Rowen_Vishne_2021} G. Blachar, E. Matzri, L. Rowen, and U. Vishne. $l$-weak identities and central polynomials for matrices. {\it In Polynomial Identities in Algebras, Springer INdAM Series},
44:69-95, 2021.

\bibitem{Drensky_1986} V. Drensky. Weak identities in the algebra of symmetric matrices of order two (Russian). {\it Pliska, Stud. Math. Bulg.}, 8:77-84, 1986.

\bibitem{Drensky_1997} V. Drensky. Identities of representations of nilpotent Lie algebras. {\it Commun. Algebra}, 25(7):2115-2127, 1997.

\bibitem{Drensky_2021} V. Drensky. Weak polynomial identities and their applications. {\it Commun. 
Math.}, 29:291-324, 2021.

\bibitem{Drensky_1987} V. Drensky and P. Koshlukov. Weak polynomial identities for a vector space with
a symmetric bilinear form. {\it Mathematics and Education in Mathematics, Proc. 16th
Spring Conf., Sunny Beach/Bulg.}, pages 213-219, 1987.

\bibitem{Drensky_Rashkova_1993} V. Drensky and T. Rashkova. Weak polynomial identities for the matrix algebras. {\it Commun. Algebra}, 21(10):3779-3795, 1993.

\bibitem{Askar_2004} A. Dzhumadil’daev. $N$-commutators. {\it Comment. Math. Helv.}, 79(3):516-553, 2004.

\bibitem{Askar_2014} A. Dzhumadil’daev. $2p$-commutator on differential operators of order $p$. {\it Lett. Math. Phys.}, 104(7):849-869, 2014.

\bibitem{Askar_Yeliussizov_2015} A. Dzhumadil’daev and D. Yeliussizov. Path decompositions of digraphs and their applications to Weyl algebra. {\it Adv. in Appl. Math.}, 67:36-54, 2015.

\bibitem{Fideles_Koshlukov_2023_JA} C. Fideles and P. Koshlukov. $\mathbb{Z}$-graded identities of the Lie algebras $U_1$. {\it J. Algebra}, 633:668-695, 2023.

\bibitem{Fideles_Koshlukov_2023_Camb} C. Fidelis and P. Koshlukov. $\mathbb{Z}$-graded identities of the Lie algebras $U_1$ in characteristic
$2$. {\it Math. Proc. Camb. Phil. Soc.}, 174(1):49-58, 2023.

\bibitem{W1_2015} J. Freitas, P. Koshlukov, and A. Krasilnikov. $\mathbb{Z}$-graded identities of the Lie algebra $W_1$. {\it J. Algebra}, 427:226-251, 2015.

\bibitem{Isaev_Kislitsin_2013} I. Isaev and A. Kislitsin. Identities in vector spaces and examples of finite-dimensional linear algebras having no finite basis of identities. {\it Algebra and Logic}, 52(4):290-307, 2013.

\bibitem{Isaev_Kislitsin_2017} I. Isaev and A. Kislitsin. Identities in vector spaces embedded in finite associative algebras. {\it J. Math. Sciences}, 221(6):849-856, 2017.

\bibitem{Kislitsin_2022} A. Kislitsin. Minimal nonzero L-varieties of vector spaces over the field $\mathbb{Z}_2$. {\it Algebra and Logic}, 61(4):313-317, 2022.

\bibitem{Koshlukov_1997} P. Koshlukov. Weak polynomial identities for the matrix algebra of order two. {\it J. Algebra}, 188:610-625, 1997.

\bibitem{Koshlukov_1998} P. Koshlukov. Finitely based ideals of weak polynomial identities. {\it Commun. Algebra}, 26(10):3335-3359, 1998.

\bibitem{Lopatin_Rodriguez_2022} A. Lopatin and C. A. Rodriguez Palma. Identities for a parametric Weyl algebra over a ring. {\it J. Algebra}, 595:279-296, 2022.

\bibitem{Lopatin_Rodriguez_II} A. Lopatin and C. A. Rodriguez Palma. Identities for subspaces of a parametric Weyl algebra. {\it Lin. Algebra Appl.}, 654:250-266, 2022.

\bibitem{Lopatin_Rodriguez_III} A. Lopatin and C. A. Rodriguez Palma. Identities for subspaces of the Weyl algebra. {\it Commun. Math.}, 32(2):111-125, 2024.

\bibitem{Lopatin_Shestakov_2013} A. Lopatin and I. Shestakov. Associative nil-algebras over finite fields. {\it Inter. J. Algebra Comput.}, 23(8):1881-1894, 2013.

\bibitem{Ma_Racine_1990} W. Ma and M. Racine. Minimal identities of symmetric matrices. {\it Trans. Amer. Math. Soc.}, 320(1):171-192, 1990.

\bibitem{Razmyslov_1973_AL} Yu. Razmyslov. Finite basing of the identities of a matrix algebra of second order over a field of characteristic zero. {\it Algebra and Logic}, 12(1):47-63, 1973.

\bibitem{Razmyslov_1973_USSR} Yu. Razmyslov. On a problem of Kaplansky. {\it Math. USSR, Izv.}, 7(3):479-496, 1973.

\bibitem{Razmyslov_book} Yu. Razmyslov. Identities of algebras and their representations, volume 138 of {\it Transl. Math. Monogr.} Amer. Math. Soc., Providence, RI, 1994.

\end{thebibliography}
% Please, do not change the above line and do not insert your references
% into this file.  Instead, insert your references into the cimart.bib file.
% See cimart.bib for further instructions.

\EditInfo{ February 23, 2024}{February 26, 2024}{Ivan Kaygorodov}

\end{document}